\documentclass[a4paper,  leqno,  myheadings,  11pt,  twoside]{article}
\usepackage{mathrsfs}
\usepackage{hyperref}

\usepackage{amsmath,  amsthm,  amssymb,  amscd,  amsxtra, graphicx}
\usepackage{latexsym,  amsfonts}

\DeclareMathOperator*{\esssup}{esssup}

\newcommand{\norm}[1]{\lVert#1\rVert}
\newcommand{\lnorm}[1]{{\lVert#1\rVert}_\infty}
\def\qc{quasi\-conformal }

\def\md{\mathbb{D}}
\def\mc{\mathbb{C}}
\def\msc{\mathscr{C}}
\def\msu{\mathscr{U}}
\def\mcn{\mathcal{N}}
\def\mcc{\mathcal{C}}

\def\mn{\mathbb{N}}

\def\T{Teich\-m\"ul\-ler }
\def\nde{non-decreasable }
\def\dcs{decreasable }

\def\nmu{\|\mu\|_\infty}

\def\nmu{\|\mu\|_\infty}

\def\wt{\widetilde}
\def\vp{\varphi}

\def\vp{\varphi}

\def\limn{\lim_{n\to\infty}}

\def\ov{\overline}

\def\de{\Delta}
\def\q1s{Q^1(S)}

\def\zs{Z(S)}

\def\qqde{Q(\de)}

\def\zde{Z(\de)}

\def\boz{[0]_Z}
\def\bmu{[\mu]_Z}
\def\bnu{[\nu]_Z}

 \makeatletter
\@addtoreset{equation}{section}

\newtheorem{theorem}{Theorem}
\newtheorem{cor}{Corollary}
\newtheorem{lemma}{Lemma}[section]

\newtheorem{theo}{Theorem}

\newtheorem*{question}{Question}

\newtheorem*{note}{Note}

\newtheorem{theoa}{Theorem A'\!\!}
\renewcommand{\thetheoa}

\newtheorem{theor}{Construction Theorem\!\!}
\renewcommand{\thetheor}

\begin{document}

\title{\bf{Non-uniqueness of infinitesimally  weakly \nde extremal dilatations}
\author{GUOWU YAO
\\
Department of Mathematical Sciences,
  Tsinghua University\\ Beijing,  100084,
   People's Republic of
  China \\E-mail: \texttt{wallgreat@tsinghua.edu.cn}
}}
 %\date{July 29,  2015}
 %\date{}
\maketitle
\begin{abstract}\noindent
In this paper, it is shown that a weakly \nde dilatation in an infinitesimal \T equivalence class can be not  a \nde one.  As an application, we prove that if an infinitesimal equivalence class contains more than one extremal dilatation, then it contains infinitely many weakly \nde extremal dilatations.
%It has shown by the author that  a weakly \nde extremal dilatation is not necessarily non-decreasable.
\end{abstract}
\renewcommand{\thefootnote}{}
%\footnote{$^*$Corresponding author.}

\footnote{Keywords: \T space,
\qc map, weakly non-decreasable,  non-decreasable.}

\footnote{2010 \textit{Mathematics Subject Classification.} Primary
30C75;  30C62.}
\footnote{The  work was  supported by   the National Natural Science Foundation of China (Grant
No. 11271216).}

\section{\!\!\!\!\!{. }
Introduction}\label{S:intr}

Let $S$ be a plane domain with at least two boundary points.
Denote by $Bel(S)$ the Banach space of Beltrami differentials
$\mu=\mu(z)d\bar z/dz$ on $S$ with finite $L^{\infty}$-norm.

Let
$Q(S)$ be the Banach space  of integrable holomorphic quadratic differentials on $S$ with
$L^1-$norm
\begin{equation*}
\|\vp\|=\iint_{S}|\vp(z)|\, dxdy<\infty.
 \end{equation*}

Two Beltrami differentials $\mu$ and $\nu$ in $Bel(S)$ are said to
be infinitesimally \T equivalent if
\begin{equation*}\iint_S(\mu-\nu)\vp \, dxdy=0,  \text{
for any } \vp\in Q(S).
\end{equation*}
The infinitesimal \T space $\zs$  is defined as  the quotient space of $Bel(S)$ under the equivalence
relation.  Denote by $\bmu$ the equivalence class of $\mu$ in
$\zs$. Especially, we use $\mcn(S)$ to denote the set of Beltrami differentials in $Bel(S)$ that is equivalent to 0.

$\zs$ is a Banach space and its standard
sup-norm is defined by
\begin{equation*}
\|\bmu\|=\norm{\mu}:=\sup_{\vp\in \q1s}Re\iint_S \mu\vp \,
dxdy=\inf\{\|\nu\|_\infty:\,\nu\in\bmu\}.\end{equation*}
 We say that $\mu$ is extremal  (in $\bmu$) if $\nmu=\|\bmu\|$,
uniquely extremal if $\|\nu\|_\infty>\nmu$ for any other $\nu\in
\bmu$.

$\zs$ is a Banach space and  its standard
norm satisfies
\begin{equation*}
\|\bmu\|=\norm{\mu}:=\sup_{\vp\in \q1s}Re\iint_S \mu\vp \,
dxdy=\inf\{\|\nu\|_\infty:\,\nu\in\bmu\}.\end{equation*}
 We say that $\mu$ is extremal  (in $\bmu$) if $\nmu=\|\bmu\|$,
uniquely extremal if $\|\nu\|_\infty>\nmu$ for any other $\nu\in
\bmu$.

 A Beltrami differential $\mu$ (not necessarily extremal) is called to be  \emph{\nde} in its class $\bmu$ if  for $\nu\in \bmu$,
  \begin{equation}
  |\nu(z)|\leq|\mu(z)|\  a.e. \text{\ in } S,\end{equation}
implies that $\mu=\nu$; otherwise, $\mu$  is called to be \emph{decreasable}.

 The notion of \nde dilatation was firstly introduced by Reich in \cite{Re4} when he studied the unique extremality of \qc mappings.
 A uniquely extremal Beltrami differential is obviously non-decreasable.

 Let $\de$ denote the unit disk $\{z\in\mc:\;|z|<1\}$ and $\zde$ be  the infinitesimal \T space on $\de$.
 In \cite{SC}, Shen and Chen   proved  the following theorem.
 \begin{theo}\label{Th:sc}
For every $\mu\in Bel(\de)$, there exist infinitely many  \nde dilatations  in the  infinitesimal equivalence class $\bmu$ unless $\bmu=\boz$.
\end{theo}

 The author \cite{Yao2} proved that an infinitesimal  \T class may contain infinitely many \nde extremal dilatations. The existence of a \nde extremal in a class is generally unknown.

In \cite{ZZC}, Zhou et al. defined \emph{weakly \nde dilatation} in a \T equivalence class and proved that there always exists a weakly \nde extremal dilataion in a \T class.  Following their defintion, we say that $\mu\in Bel(S)$ is  a \emph{strongly \dcs dilatation} in $\bmu$ if there exists $\nu\in \bmu$ satisfying the following conditions:\\
(A) $|\nu(z)|\leq |\mu(z)|$ for almost all $z\in S$,\\
(B) There exists a domain $G\subset S$  and a positive number $\delta>0$ such that
\begin{equation*}
|\nu(z)|\leq |\mu(z)|-\delta, \text{ for almost all  } z\in G.
\end{equation*}
Otherwise, $\mu$ is called \emph{weakly non-decreasable}. In other words, a Beltrami differential $\mu$ is called weakly \nde if either $\mu$ is \nde or $\mu$ is decreasable but is not strongly decreasable. For the sake of mathematical precision,  we call a Beltrami differential $\mu$  to be a \emph{pseudo \nde dilatation} if it is a weakly \nde dilatation but  not a \nde dilatation.

In \cite{Yao11}, the author proved the following theorem.

\begin{theo}\label{Th:infweaknde}Suppose  $\bmu\in \zde$.  Then
 there is a weakly \nde extremal dilatation $\nu$ in $\bmu$.
\end{theo}

In the end of \cite{Yao11}, the following question is posed.
\begin{question}  Whether a weakly \nde dilatation in $\bmu$ is a \nde one?
  \end{question}
In other words, the question is equivalent to ask whether there exists a pseudo \nde dilatation in some $\bmu$.

In this paper, we answer the question negatively at first.

\begin{theorem}\label{Th:infweak0}For any given $\lambda>0$, the basepoint $\boz$ contains infinitely many  pseudo \nde dilatations $\nu$  such that $\lnorm{\nu}=\lambda$ and  the support set of  each  $\nu$ in $\de$ has empty interior.  However, $0$ is the unique \nde dilatation in $\boz$.
\end{theorem}

By further applying Theorem \ref{Th:infweak0}, we prove the following results.

\begin{theorem}\label{Th:pseudo0}Suppose  $\bmu\in \zde$ and $\lambda>\norm{\bmu}$. Then $\bmu$ contains  infinitely many weakly \nde  dilatations $\nu$ with $\lnorm{\nu}\leq\lambda$.
\end{theorem}

\begin{theorem}\label{Th:pseudo1}Let  $\bmu\in \zde$.  If the extremal in the point $\bmu\in\zde$ is not unique, then $\bmu$ contains infinitely many weakly \nde extremal dilatations.
\end{theorem}

Since Theorem 2 in \cite{Yao2} indicates that there exists $\bmu\in \zde$ such that $\bmu$ contains infinitely many extremals but only one \nde extremal, we have the following corollary.

\begin{cor}\label{Th:infweak1}
There exists a Beltrami differential  $\mu\in  Bel(\de)$ such that $\mu$ is the unique \nde extremal dilatation in $\bmu$ while $\bmu$ contains infinitely many pseudo \nde extremal dilatations.
\end{cor}

By use of some technique in \cite{Re6,Yao2,Yao3}, we can obtain the following interesting theorem.

\begin{theorem}\label{Th:infweak2}There exists an extremal Beltrami differential  $\mu\in Bel(\de)$ such that $\bmu$ contains infinitely many \nde extremal dilatations and $\bmu$ contains infinitely many  pseudo \nde extremal dilatations.
\end{theorem}

It seems that Theorem \ref{Th:pseudo0} is covered by Theorem \ref{Th:sc}. But it is not in such a case. On the one hand, $\boz$ contains infinitely many weakly \nde dilatations while $0$ is the unique \nde dilatation. On the other hand, after investigating the proof in \cite{SC}, we find that  Shen and Chen actually  proved Theorem \ref{Th:sc} in the following precise form.

\begin{theoa} Suppose  $\bmu\neq \boz$. Then for sufficiently large $\lambda>\norm{\bmu}$, $\bmu$ contains  infinitely many  \nde  dilatations $\nu$ with $\lnorm{\nu}\geq\lambda$.
\end{theoa}

%After doing some preparations in Section \ref{S:prepar}, we prove Theorems \ref{Th:pseudo0} and \ref{Th:pseudo1} in Section \ref{S:proof}.

\section{\!\!\!\!\!{. }
Some preparations }\label{S:prepar}

The first lemma comes Lemma 2.2 in \cite{Yao11}.
\begin{lemma}\label{Th:cor}Suppose that $\mu\in Bel(\de)$. Let $\alpha\in Bel(\de)$ Then for any $z_0\in \de$ and $\epsilon>0$,  there exists  $\nu\in \bmu$ and a small $r>0$ such that $\lnorm{\nu|_{\de\backslash\de(\zeta,r)}}\leq \lnorm{\mu}+\epsilon$ and
\begin{equation}\label{Eq:mug0} \nu(z)=\alpha(z),\quad \text{ when } z\in \de(\zeta,r)=\{z:\;|z-\zeta|<r\}.\end{equation}
 In particular, $\nu$ vanishes on $\de(\zeta,r)$ when $\alpha=0$.
 In particular, $\nu$ vanishes on $\de(\zeta,r)$ when $\alpha=0$.
\end{lemma}

The second lemma is actually Theorem 2 in \cite{Yao11}.
\begin{lemma}\label{Th:lempseudo}Suppose  $\bmu\in \zde$.  Let $\chi\in \bmu$ and  $\msu=\{\alpha\in \bmu:\; |\alpha(z)|\leq |\chi(z)| \text{ a.e. on } \de\}$. Then
then there is a weakly \nde dilatation $\nu$ in $\msu$.
\end{lemma}

\begin{lemma}\label{Th:deform2}Let  $J_i\subset \de $ ($i=1,2,\ldots,m$) be
$m$ Jordan domains such that $\ov{J_i}\subset \de$, $\ov{J_i}$ ($i=1,2,\ldots,m$) are
mutually disjoint and $\de\backslash  \bigcup^m_1 \ov{J_i}$ is
connected. Suppose $\mu,\nu\in Bel(\de)$ satisfying $\mu(z)=\nu(z)$ a.e. on $\de\backslash \bigcup_{i=1}^m \ov{J_i}$. Then the following two conditions is equivalent:
\\
(a) $\bmu=\bnu$,\\
(b) $[\mu|_{J_i}]_Z=[\nu|_{J_i}]_Z$, where we regard $[\mu|_{J_i}]_Z, [\nu|_{J_i}]_Z$ as the points in the infinitesimal \T space $Z(J_i)$, ($i=1,2,\cdots,m$).
\begin{proof}See the proof of Lemma 3 of \cite{Yao3}.

\end{proof}

\end{lemma}

For $\mu\in Bel(\de)$, $\vp\in\qqde$, let
 \begin{equation*}
 \lambda_\mu[\vp]=Re\iint_\de \mu(z)\vp(z)dxdy.
 \end{equation*}

The following   Construction Theorem is essentially due to Reich \cite{Re6} and  is very useful for the study of (unique) extremality of \qc mappings (see \cite{Re6,Yao,Yao2,Yao3}).

\begin{theor}\label{Th:Re1}
Let $A$ be a compact subset of $\de$ consisting  of $m$
($m\in\mathbb{N}$) connected components and such that $\de\backslash
A$ is connected and  each connected component of $A$ contains at
least two points. There exists a function $\mathcal{A}\in\L^\infty(\de)$
and a sequence $\vp_n\in \qqde$ $(n=1,2,\ldots)$ satisfying the
following conditions $(\ref{Eq:q1})-(\ref{Eq:q4})$:
\begin{equation}\label{Eq:q1}
|\mathcal{A}(z)|=\begin{cases} 0, \quad &z\in A, \\
1,\quad  & \;for\; a.a. \;z\in \de\backslash A,\end{cases}
\end{equation}
\begin{equation}\label{Eq:q2}\limn\{\|\vp_n\|-\lambda_\mathcal{A}[\vp_n]\}=0,
\end{equation}
\begin{equation}\label{Eq:q3}\limn|\vp_n(z)|=\infty  \quad a.e. \; in\;
\de\backslash A.
\end{equation} and as $n\to\infty$,
\begin{equation}\label{Eq:q4}\vp_n(z) \to0  \text{  uniformly on }  A.
\end{equation}

\end{theor}
\begin{proof} See the proof of
 of Construction Theorem  in \cite{Yao3}.
\end{proof}

From the Construction Theorem, we can get
\begin{lemma}\label{Th:lemext}
Let $A$  be as in Construction Theorem and
$\mathcal{A}(z)$ be constructed by
 Construction Theorem.
Let
\begin{equation*}
\nu(z)=\begin{cases}k\mathcal{A}(z),\quad &z\in \de\backslash  A,\\
\mathcal{B}(z), \quad & z\in A,
\end{cases}
\end{equation*}
where $k<1$ is a positive
constant and $\mathcal{B}(z)\in L^\infty(A)$ with $\lnorm{\mathcal{B}}\leq k$.
Then
$\nu(z)$ is   extremal in $[\nu]$ and for any $\chi(z)$ extremal in
$\bnu$, $\chi(z)=\nu(z)$ for almost all $z$ in
$\de\backslash A$.
\end{lemma}
\begin{proof} See the proof of
 of Lemma 5 in \cite{Yao3}.
\end{proof}

\begin{lemma}\label{Th:lemext1}Let  $J_i\subset \de $ ($i=1,2,\ldots,m$) be
$m$ Jordan domains such that $\ov{J_i}\subset \de$, $\ov{J_i}$ ($i=1,2,\ldots,m$) are
mutually disjoint and $\de\backslash  \bigcup^m_1 \ov{J_i}$ is
connected. Put $A=\bigcup^m_1 \ov{J_i}$. Let $\mathcal{A}(z)$   be constructed by the Construction
Theorem.
Let
\begin{equation*}
\nu(z)=\begin{cases}k\mathcal{A}(z),\quad &z\in \de\backslash  A,\\
\mathcal{B}(z), \quad & z\in A,
\end{cases}
\end{equation*}
where $k<1$ is a positive
constant and $\mathcal{B}(z)\in L^\infty(A)$ with $\lnorm{\mathcal{B}}\leq k$.
  We regard $[\nu|_{J_i}]_Z$ as  a point in the \T space $T(J_i)$, $i=1,2,\ldots,m$. Then,\\
(A) $\nu$ is a weakly \nde  dilatation in $\bnu$ if and only if every $\nu|_{J_i}$ is a weakly \nde dilatation in $[\nu|_{J_i}]_Z$, $i=1,2,\ldots,m$;\\
(B)  $\nu$ is a \nde  dilatation in $\bnu$ if and only if every $\nu|_{J_i}$ is  a \nde dilatation in $[\nu|_{J_i}]_Z$, $i=1,2,\ldots,m$.
\end{lemma}
\begin{proof} It is evident that $\nu$ is extremal in $\bnu$ by Lemma \ref{Th:lemext}.

(A) The ``only if" part is obvious. Now, assume that every $\nu|_{J_i}$ is a  weakly \nde  dilatation in $[\nu|_{J_i}]_Z$, $i=1,2,\ldots,m$. We show that $\nu$ is a weakly \nde  dilatation in $\bnu$. Suppose to the contrary. Then $\bnu$ is a strongly decreasable dilatation in $\bnu$. That is,
there exists a Beltrami differential $\eta\in\bnu$  such that \\
(1) $|\eta(z)|\leq |\nu(z)|$ for almost all $z\in \de$,\\
(2) there exists a domain $G\subset \de$  and a positive number $\delta>0$ such that
\begin{equation*}
|\eta(z)|\leq |\nu(z)|-\delta, \text{ for almost all  } z\in G.
\end{equation*}
Observe that $\eta$ is  extremal in $\bnu$ and hence $\eta(z)=\nu(z)$ a.e on $\de\backslash A$ by Lemma \ref{Th:lemext}.
It forces that  $G$ is contained in some $J_i$. Furthermore, by Lemma \ref{Th:deform2}, we have $\eta|_{J_i}\in [\nu|_{J_i}]_Z$. Thus $\nu|_{J_i}$ is a strongly decreasable dilatation in $[\nu|_{J_i}]_Z$, a contradiction.

(B) The ``only if" part is also obvious. Assume that every $\nu|_{J_i}$ is a    \nde  dilatation in $[\nu|_{J_i}]_Z$, $i=1,2,\ldots,m$. We show that $\nu$ is a   \nde  dilatation in $\bnu$. Suppose to the contrary. Then $\bnu$ is a  decreasable dilatation in $\bmu$. There exists a Beltrami differential $\eta\in\bnu$  such that  $|\eta(z)|\leq |\nu(z)|$ for almost all $z\in \de$ but $\eta(z)\neq \nu(z)$ on a subset $E\subset\de$ with positive measure. It is no harm to assume that $E\cap J_1$ has positive measure.   Since $\eta(z)=\nu(z)$ a.e on $\de\backslash A$, it follows from Lemma \ref{Th:deform2} that   $\eta|_{J_1}\in [\nu|_{J_1}]_Z$.  Thus, $\nu|_{J_1}$ is decreasable in $[\nu|_{J_1}]_Z$, a contradiction.

\end{proof}

 The following lemma comes from Lemma 7 in \cite{Yao2}.

\begin{lemma}\label{Th:lemext2}Set  $\de_s=\{z:\,|z|<s\}$ for $s\in (0,1)$.
Let $\chi(z)$ be defined as follows,
\begin{equation}\label{Eq:des}
\chi(z)=\begin{cases}0, \, & z\in \de-\de_s,\\
\wt k\,&z\in \de_s,
\end{cases}
\end{equation}
where $\wt k<1$ is a positive constant. Then $[\chi]_Z$  contains infinitely many
\nde Beltrami differentials $\eta$ with $\lnorm{\eta}<\wt k$.
\end{lemma}

\section{\!\!\!\!\!{. }
Proofs of Theorems \ref{Th:infweak0} and \ref{Th:infweak2}}\label{S:proof1}

\textbf{Proof of Theorem \ref{Th:infweak0}.}
Let $\mcc$ be a compact subset with empty interior and positive measure $meas(\mcc)\in (0,1)$.
Put  $\msc=\{re^{i\theta}:\; r\in \mcc,\; \theta\in [0,2\pi)\}$.
Then $\msc$ is 2-dimensional compact subset in $\de$ with empty interior and  $meas(\msc)\in (0,\pi)$.

Note that $\mcn(\de)=\boz$.  It is obvious that $0$ is the unique \nde dilatation in $\mcn(\de)$. We now show that,  for any given $\lambda>0$, $\mcn(\de)$ contains infinitely many  pseudo \nde dilatations $\nu$ with  $\lnorm{\nu}=\lambda$ and  the support set of  each  $\nu$ in $\de$ has empty interior.
 Fix a positive integer number $m$ and let
\begin{equation*}\gamma(z)=\begin{cases} \lambda \frac{z^m}{|z|^m},\;&z\in \msc,\\
0, \;&z\in \de\backslash \msc.
\end{cases}\end{equation*}

\textit{Claim.} $\gamma\in \mcn(\de)$ and is a pseudo \nde dilatation in $\boz$.

By the definition of $\mcn(\de)$, we need to show  that
\begin{equation*}
\iint_\md \gamma(z)\vp(z)\;dxdy=0,\text{ for any } \vp\in \qqde.\end{equation*}
Note that $\{1,z,z^2,\cdots, z^n,\cdots\}$ is a base of the Banach space $\qqde$. It suffices to prove
\begin{align}\label{Eq:zzn}
\iint_\de \gamma(z) z^n\;dxdy=0,\; \text{ for any } n\in \mn.
\end{align}
By the definition of $\mcc$, we see that the open set $\mathscr{A}=[0,1]\backslash \mcc$ is the union of countably many disjoint open intervals.  Set $\mathscr{D}=\{re^{i\theta}:\; r\in \mathscr{A},\; \theta\in [0,2\pi)\}$. It is clear that
$\de=\mathscr{D}\cup \msc$ and
\begin{align}\label{Eq:zzninf}
\iint_\de \gamma(z) z^n\;dxdy=\iint_{\mathscr{D}} \gamma(z) z^n\;dxdy,\; \text{ for any } n\in \mn.
\end{align}

Observe that $\mathscr{D}$ is the union of  countably many disjoint ring domains each of which can be written in the form $R=\{re^{i\theta}:\; r\in (x,x'),\; \theta\in [0,2\pi)\}$, $x,x'\in (0,\lambda)$.
A simple computation shows that
\begin{equation}\label{Eq:zzn1}
\begin{split}
&\iint_R \gamma(z) z^n\;dxdy=\iint_R \lambda\frac{z^m}{|z|^m} z^n\;dxdy\\
&=\lambda\int_0^1 r^{n+1}\;dr\int_{0}^{2\pi}e^{i(m+n)\theta}\;d\theta=0,\; \text{ for any } n\in \mn.
\end{split}\end{equation}
Hence, we get
\begin{align}\label{Eq:zzn3}
\iint_\mathscr{D} \gamma(z) z^n\;dxdy=0,\; \text{ for any } n\in \mn.
\end{align}
Thus, we have prove that $\gamma\in \mcn(\de)$. Since the support set of $\gamma$ in $\de$  has empty interior, by the definition $\gamma$ is a weakly \nde dilatation in $\boz$. On the other hand, it is obvious that  0 is the unique  \nde dilatation in $\boz$ and hence  $\gamma$ is a pseudo \nde dilatation.
  When $m$ varies over $\mn$ or the set $\mcc$ varies suitably, we obtain infinitely many pseudo \nde dilatations in $\boz$.
The completes the proof of Theorem \ref{Th:infweak0}.

\textbf{Proof of Theorem \ref{Th:infweak2}.}
Choose
$J_1=\{z\in \de:\;|z|<\frac{1}{4}\}$ and $J_2=\{z\in \de:\;|z-\frac{1}{2}|<\frac{1}{8}\}$. Let $A=\ov {J_1}\cup \ov {J_2}$. Let
$\mathcal{A}(z)$  be constructed by
the  Construction Theorem and let $\mu(z)=k\mathcal{A}(z)$ where
$k<1$ is a positive constant.  Let $\de_s=\{z\in  \de:\;|z|<s\}$ where $s\in (0,\frac{1}{4})$ and let $\wt k\in (0,k]$ be a constant.  Set
\begin{equation*}
\mu(z)=\begin{cases}k\mathcal{A}(z),\,&z\in \de\backslash  A,\\
\wt k\,&z\in \de_s,\\
0, \, & z\in
J_1-\de_s,\\
0, \, & z\in
J_2.
\end{cases}
\end{equation*}

By Theorem \ref{Th:infweak0}, $[0|_{J_2}]_Z$ contains infinitely many pseudo \nde dilatations and $0|_{J_2}$ is the unique \nde dilatation in $[0|_{J_2}]_Z$.
Applying  Lemma \ref{Th:lemext2} to $J_1$, we see that $[0|_{J_1}]_Z$ contains infinitely many  \nde dilatations with $L^\infty-$norm of at most $k$. By the foregoing reason, it derives readily  from Lemma \ref{Th:lemext1} that $\bmu$ is the desired \T class. The gives Theorem \ref{Th:infweak2}.

\section{\!\!\!\!\!{. }
Proofs of Theorems \ref{Th:pseudo0} and \ref{Th:pseudo1}}\label{S:proof}
To make the proof more concise, we  prove a new theorem from which Theorems \ref{Th:pseudo0} and \ref{Th:pseudo1} follows readily. We introduce the conception of non-landslide at first.

A  Beltrami differential $\mu$ (not necessarily extremal) in $Bel(S)$ is  said to be
 \emph{landslide } if there exists a non-empty open subset $G\subset S$ such
that
\begin{equation*}\esssup_{z\in
G}|\mu(z)|<\|\mu\|_\infty;\end{equation*} otherwise, $\mu$ is said
to be of \emph{non-landslide}.

Throughout the section, we denote by $\de(\zeta,r)$ the round disk $\{z:\;|z-\zeta|<r\}$ ($r>0$) and let $\msu_\lambda=\{\alpha\in \bmu:\;\lnorm{\alpha}\leq \lambda\}$.

\begin{theorem}\label{Th:pseudo2}Suppose  $\bmu\in \zde$ and $\lambda\geq\norm{\bmu}$.
If $\msu_\lambda$ contains a landslide Beltrami differential, then  $\msu_\lambda$ contains  infinitely many weakly \nde  dilatations.
\end{theorem}

\begin{proof}
If  $\bmu=\boz$ and  $\lambda=0$, nothing needs to prove. Now let
 $\lambda>0$.

 Let $\alpha$ be a landslide dilatation in $\msu_\lambda$. Then there is $\lambda'\in (0, \lambda)$ and a sub-domain  $G\subset \de$ such that
  $|\alpha(z)|\leq \lambda'$ on $G$.  Applying Lemma \ref{Th:cor} on $G$, we can find a Beltrami differential $\chi\in \msu_\lambda$ such that $\chi(z)=0$ on some small disk $\de(\zeta,\rho)\subset G$.

  By Lemma \ref{Th:lempseudo}, we can find a weakly \nde dilatation  $\nu\in \msu_\lambda$ such that $|\nu(z)|\leq |\chi(z)|$ a.e. on $\de$. It is obvious that $\nu(z)=0$ on $\de(\zeta, \rho)$.

  Let $D=\de(\xi, r)$ be a small round disk in $\de(\zeta,\rho)$, $r\in (0,\rho)$.  We regard $[0|_D]_Z$ as the basepoint in the infinitesimal \T space $Z(D)$. By Theorem \ref{Th:infweak0}, we may choose a pseudo \nde dilatation $\gamma\neq 0|_D$ in
  $[0|_D]_Z$ whose support set in $D$ has empty interior such that $\lnorm{\gamma}\leq \lambda$. Put
  \begin{equation*}\label{Eq:pps1}
\beta(z)=\begin{cases}\nu(z),\quad &z\in \de\backslash  \ov D,\\
\gamma(z), \quad & z\in D.
\end{cases}
\end{equation*}

If $\beta$ is a weakly \nde dilatation in $\msu_\lambda$, then let $\nu_D=\beta$.  In fact, $\nu_D$ is necessarily a pseudo \nde dilatation since $\beta|_D=\gamma$ is decreasble   but is not strongly decreasable in $[0|_D]_Z$.

Otherwise,   $\beta$ is not a weakly \nde dilatation in $\msu_\lambda$, and then by the proof of Lemma \ref{Th:lempseudo} there is a weakly \nde dilatation $\beta'$ in $\msu_\lambda$ such that \\
(1) $|\beta'(z)|\leq |\beta(z)|$,\\
(2) there exists a small round disk $\de(z',r')\subset \de$  and a positive number $\delta>0$ such that
\begin{equation*}
|\beta'(z)|\leq |\beta(z)|-\delta, \text{ for almost all  } z\in \de(z',r').
\end{equation*}
Since  the support set of $\beta$ on $\de(\zeta,\rho)$ has empty interior, it forces that $\de(z',r')\subset \de\backslash \de(\zeta,\rho)$.
Let $\nu'_D=\beta'$.

\textit{Claim.} $\beta'|_D\not\in [0|_D]_Z$ where $\beta'|_D$ is the restriction of $\beta'$ on $D$.

Suppose to the contrary. Then $\beta'|_D\in [0|_D]_Z$.
 Let
\begin{equation*}\label{Eq:ps1}
\wt\beta'(z)=\begin{cases}\beta'(z),\quad &z\in \de\backslash  \ov D,\\
0, \quad & z\in D.
\end{cases}
\end{equation*}
It is clear that $\wt\beta'\in \msu_\lambda$. Since $\beta(z)=\beta'(z)=0$ on  $\de(\zeta, \rho)\backslash D$, it is easy to verify that\\
(1) $\wt \beta'(z)=\nu(z)=0$ on $\de(\zeta,\rho)$,\\
(2) $|\wt\beta'(z)|\leq |\nu(z)|$,\\
(3) $|\wt\beta'(z)|<|\nu(z)|-\delta$ on $\de(z',r')$.\\
Thus, $\nu$ is strongly decreasable on $\de$, a contradiction. The claim is proved.

For convenience, let $\wt \nu_D$ denote either $\nu_D$ or $\nu'_D$.

 Now, let $\{D_n=\de(z_n,r_n)\}$ be a sequence of round disks in $\de(\zeta,\rho)$ which are mutually disjoint and $\{\gamma_n\in [0|_{D_n}]_Z:\;\gamma_n\neq 0|_{D_n}\}$ be a sequence of Beltrami differentials whose support sets have empty interior in $D_n$ respectively. By the previous analysis, we get either a pseudo  \nde dilatation $\nu_{D_n}$ or a weakly \nde dilatation $\nu'_{D_n}$ in $\msu_\lambda$. It is easy to check that whenever $n\neq m$, it holds that  $\wt \nu_{D_n}\neq \wt\nu_{D_m}$.
Thus, we get infinitely many weakly \nde dilatations  in $\msu_\lambda$.
\end{proof}

\begin{note} If the weakly \nde dilatation in  $\msu$
  is unique, then it is necessarily a \nde dilatation in $\bmu$.
  \end{note}

\textbf{Proof of Theorems \ref{Th:pseudo0} and \ref{Th:pseudo1}}. At first, let $\lambda>\norm{\bmu}$. It is evident that $\bmu$ contains a landslide Beltrami differential. Then by Theorem \ref{Th:pseudo2}, there are infinitely many weakly \nde dilatations in $\msu_\lambda$. The gives Theorem \ref{Th:pseudo0}.

Secondly, let $\lambda=\norm{\bmu}$. Then $\msu_\lambda$ just contains all extremal dilatations  in $\bmu$.  Now assume  $\#\msu_\lambda>1$. Then $\bmu$ contains infinitely many extremal dilatations.

\textit{Case 1.} Each extremal dilatation in $\bmu$ is non-landslide.

By the way, it is an open problem whether there exists $\bmu$ such that the extremal in $\bmu$ is not unique and each extremal in  $\bmu$ is non-landslide.
Anyway, by definition each extremal in $\bmu$ is weakly non-decreasable, and hence
 $\bmu$ contains either  infinitely many \nde extremal dilatations or infinitely many pseudo \nde extremal dilatations.

 \textit{Case 2.} $\bmu$ contains a landslide extremal dilatation.

 Then by Theorem \ref{Th:pseudo2}, there are infinitely many weakly \nde dilatations in $\msu_\lambda$.

 Theorem \ref{Th:pseudo1} now follows.

\renewcommand\refname{\centerline{\Large{R}\normalsize{EFERENCES}}}

\end{document}